\documentclass[12pt]{amsart}

\usepackage{amsthm}
\usepackage{amsmath}
\usepackage{amsfonts,amssymb,latexsym}
\usepackage[english]{babel}
\usepackage{epsfig,epsf}
\usepackage{color}
\newtheorem{thm}{Theorem}
\newtheorem{lem}{Lemma}
\newtheorem{cor}{Corollary}

\renewcommand{\thetheirtheorem}

\newcommand{\Acal}{{\mathcal A}}

\newcommand{\Scal}{{\mathcal S}}

\newcommand{\Znat}{{\mathbb Z}}
\newcommand{\Z}{{\mathbb Z}}

\newcommand{\R}{{\mathbb R}}

\title[Sidon]{A numerical note on upper bounds for $B_2[g]$ sets}
\author{Laurent Habsieger \and Alain Plagne }
\thanks{Both authors are supported by the ANR grant C\ae sar, number ANR 12 - BS01 - 0011.}

\email{habsieger@math.univ-lyon1.fr}
\address{Universit\'e de Lyon, CNRS UMR 5208, Universit\' e Claude Bernard Lyon 1,
Institut Camille Jordan, 43 boulevard du 11 novembre 1918, 69622 Villeurbanne Cedex, France \bigskip}

\email{plagne@math.polytechnique.fr}
\address{Centre de Math\' ematiques Laurent Schwartz, \' Ecole polytechnique, CNRS, Universit\' e Paris-Saclay, 91128 Palaiseau Cedex, France}

\begin{document}

\begin{abstract}
Sidon sets are those sets such that the sums of two of its elements never coincide. 
They go back to the 30s when Sidon asked for the maximal size of a subset of consecutive integers with that property. 
This question is now answered in a satisfactory way. Their natural generalization, called $B_2 [g]$ sets and defined 
by the fact that there are at most $g$ ways (up to reordering the summands) to represent a given integer as a sum of 
two elements of the set, are much more difficult to handle and not as well understood. 

In this article, using a numerical approach, we improve the best upper estimates on the size of a $B_2 [g]$ set in an interval of integers 
in the cases $g=2,3,4$ and $5$. 
\end{abstract}

\maketitle

\section{Introduction}

Let $g$ be a positive integer. A set $\Acal$ of integers is said to be a $B_2 [g]$ set if for any integer $n$, there are at most $g$ ways 
to represent $n$ as a sum $a+b$ with $a,b\in\Acal$ and  $a\le b$. As is usual, we denote by $F(g,N)$ the largest possible size of 
a $B_2 [g]$ set contained in $\{0,1,\dots,N\}$. In the study of $B_2 [g]$ sets, this is the most studied aspect.

These sets have a long history going back to Sidon \cite{Sidon} in the 30s and to the seminal work of Bose, Chowla and Singer 
as for lower bounds, and of Erd\H{o}s and Tur\'an as for upper bounds. See \cite{Singer,Bose,Chowla,ET}. Except in the case $g=1$, for which it is known that
$$
F(1,N) \sim \sqrt{N},
$$
the precise asymptotic behaviour of $F(g,N)$ remains unknown.

However, for any positive $g$, since a Sidon set is in particular a $B_2 [g]$ set, it is known that the quantity $F(g,N)$ grows at least 
like a constant times $\sqrt{N}$. Better lower bounds were obtained in \cite{P,HP, CRT,MOB2}. As for upper bounds, 
the current best result is %the following
\begin{equation}
\label{upperbound1}
F(g,N)\lesssim \sqrt{ \min \big( 3.1694\ g, 1.74217\ (2g-1) \big)\,  N},
\end{equation}
the first argument in the minimum being contained in \cite{M-OB} and the second one in \cite{Yu2}. Notice that the first is better 
than the second as soon as $g \ge 6$. Notice finally that it is not even known whether there is a constant $c_g$ such that 
$F(g,N) \sim c_g \sqrt{N}$.

In this article, we slightly improve on the upper bound \eqref{upperbound1} for $g<6$ by proving the following result.

\begin{thm} 
\label{theo} 
One has
$$
F(g,N)\lesssim \sqrt{1.740463\ (2g-1)\ N}.   				%		 \sqrt{1.740462703719317\dots\ (2g-1)\ N}. 
$$
\end{thm}

For instance, we obtain $F(2,N) \lesssim 2.2851 \sqrt{N}$ instead of Yu's $2.2864 \sqrt{N}$. 
These improvements remain modest but one must remember that this is the case for all the recent ones since the beginning of the 2000s
when it was proved \cite{Green} that $F(2,N) \lesssim 2.2913 \sqrt{N}$. 
Theorem \ref{theo} gives in particular a new best result in the cases $g=2,3,4$ and $5$ and corresponds to pushing Yu's method 
to some kind of extremity since it is not at all clear that the method is even able to prove $F(g,N)\lesssim \sqrt{1.74\ (2g-1)\ N}$ 
(although we conjecture that this is the case).

We examine Yu's method and use an approach similar to the one used in \cite{H} which led to new bounds 
for the dual problem of additive bases. 
In order to prove Theorem \ref{theo}, we first reformulate Yu's method \cite{Yu1, Yu2} by giving a general explicit result (Theorem \ref{Maj}) 
depending on the choice of a fixed auxiliary function. This allows us to apply the method not only to the case of polynomials (of high degree) 
but directly to the case of power series. In this frame, Yu's bound corresponds to an appropriate choice of 
the auxiliary function. Finally, we optimize the use of our general result by computing numerically a best possible 
function of a certain type. This leads to Theorem \ref{theo}.

\section{The method}
\label{sec2}

Let $\Acal$ be a set of integers contained in $\{0,1,\dots,N\}$. We define the function
$$
\hat{f}  (t) =  \sum_{a \in \Acal} \exp ( 2 \pi i at ).
$$
In particular, $\hat{f}  (0) = | \Acal |$.
 
If $d$ denotes the function counting the number of representations of an integer as a difference in $\Acal - \Acal$, 
namely, for $n \in \Znat$,
$$
d(n)= |\{(a,b)\in\Acal^2\ :\ a-b=n\} |,
$$
then we may compute that
\begin{equation}
\label{ffbar}
\vert \hat{f}  (t) \vert^2 =  \sum_{a,a' \in \Acal} \exp ( 2 \pi i (a-a')t ) = \sum_{\vert n\vert \leq N} d(n) \exp ( 2 \pi i n t )= 
 \sum_{\vert n\vert \leq N} d(n) \cos ( 2 \pi  n t ),
\end{equation}
where we use the parity of $d$ which follows from the symmetry of $\Acal - \Acal$ as multiset.

In this article, a function $b$ will be called {\em admissible} if the following holds: its set of definition $\Scal_b \subset \R$ is countable, 
symmetric with respect to zero and contains $0$, $b$ is an even function taking its values in the set of non negative real numbers $\R^+$, namely 
$b: \Scal_b \rightarrow \R^+$ and, finally,
$$
\sum_{\theta \in \Scal_b} b(\theta) < + \infty.
$$
In the sequel, for simplicity, we denote $b( \theta) = b_\theta$. 

An admissible function $b$ being chosen, we define the function $w_b$ as 
$$
w_b (t) = \sum_{\theta\in\Scal_b} b_{\theta} \exp \left( 2i\pi\theta t \right) = \sum_{\theta\in\Scal_b} b_{\theta} \cos \left( 2 \pi\theta t \right),
$$
by parity of $b$. Notice that the admissibility of $b$ implies $w_b$ to be $C^\infty (\R )$ and even.

Suppose that a subset $\Acal$ of $\{0,1,\dots,N\}$ is given, as well as an admissible function $b$, we then define
$$
D_\Acal (b) =\sum_{\vert n\vert \leq N} d(n) w_b \left(\frac{n}{N}\right).
$$
It is easy to compute, interverting the order of summations, that
\begin{eqnarray}
D_\Acal (b) 		& =	&\sum_{\vert n\vert \leq N} d(n)\sum_{\theta\in\Scal_b} b_{\theta} \exp \left( 2i\pi\theta \frac{n}{N} \right) \nonumber \\	%\sum_{\vert n\vert \leq N} d(n) w \left(\frac{n}{N}\right) 
&= &\sum_ {\theta\in\Scal_b}  b_{\theta}\sum_{\vert n\vert \leq N} d(n) \exp \left( \frac{2i\pi \theta n}{N} \right) \nonumber\\
&=&\sum_{\theta\in\Scal_b}  b_{\theta}\left\vert \hat{f}  \left( \frac{\theta}{N} \right) \right\vert^2, \label{D}
\end{eqnarray}
the last equality following from \eqref{ffbar}. This shows in particular that $D_\Acal (b)\geq b_0 |\Acal |^2 \geq 0$.

\section{Two lemmas}

In this section, we state two results which will be useful in our argument. The first one is a lemma of an analytical nature. If $w$ is an even $C^2(\R)$ function, we denote 
$$
\aligned
I_1(w) &= \int_0^1 w(t)\, {\rm d}t, \\
I_2(w) &= \int_0^1 w(t)^2\, {\rm d}t, \\
\Vert w''\Vert &= \max_{t\in [0,1]} \vert w''(t)\vert, \\
A(w) &= \vert w'(1)\vert+\Vert w''\Vert. \\
\endaligned
$$
Such an even $C^2(\R)$ function $w$ being given, we define the function $\tilde w$ as the unique $2$-periodic function coinciding 
with $w$ on $[-1,1]$.

We have the following lemma.

\begin{lem} 
\label{F} 
Let $w$ be an even $C^2(\R)$ function. For $m \in \Z$, let
$$
a_m=\int_{-1}^1 w(t)\exp(-i\pi mt)\, {\rm d} t.
$$ 
Then we have
$$
\tilde w(t)=\frac{a_0}{2}+\sum_{m=1}^{+\infty} a_m\cos(\pi mt).
$$
Moreover, the following upper bound holds		
$$
\vert a_m\vert \le \frac{2A(w)}{\pi^2m^2}\,.
$$

One also has
$$
\aligned
I_1(w) &= \frac{a_0}2, \\
I_2(w) &=  \frac{a_0^2}4 + \frac12 \sum_{m=1}^{+\infty} a_m^2, \\
2( I_2(w) -I_1(w)^2)&=\sum_{m=1}^{+\infty} a_m^2\,.
\endaligned
$$
\end{lem}

\begin{proof} Dirichlet's theorem ensures us that $\tilde w$ coincides with its Fourier expansion. This is the first equality. %avec sa s\'erie de Fourier. La majoration s'ensuit de :

As for the upper bound, we compute, using the parity of $w$,
$$\aligned
a_m &= \left\lbrack \frac{w(t)\exp(-i\pi mt)}{-i\pi m} \right\rbrack_{-1}^1
+\frac{1}{i\pi m}\int_{-1}^1 w'(t)\exp(-i\pi mt)\, {\rm d}t \\
&= \frac{1}{i\pi m}\int_{-1}^1 w'(t)\exp(-i\pi mt)\, {\rm d}t \\
%&= \frac{1}{i\pi m}  \left( \left\lbrack  \frac{w'(t)\exp(-i\pi mt)}{- i\pi m} \right\rbrack_{-1}^1 +\frac{1}{i\pi m}\int_{-1}^1 w''(t)\exp(-i\pi mt)\, {\rm d}t \right) \\
&= \frac{1}{\pi^2 m^2}\left( \left\lbrack  w'(t)\exp(-i\pi mt)\right\rbrack_{-1}^1 - \int_{-1}^1 w''(t)\exp(-i\pi mt)\, {\rm d}t \right) \\
%&=   \frac{1}{\pi^2 m^2}\left( (-1)^{m}  (w'(1)-w'(-1)) - \int_{-1}^1 w''(t)\exp(-i\pi mt)\, {\rm d}t \right)\\
&= \frac{1}{\pi^2 m^2}\left( 2 (-1)^{m} w'(1) - \int_{-1}^1 w''(t)\exp(-i\pi mt)\, {\rm d}t \right).\\
%\frac{2 (-1)^{m-1}  w'(1) }{\pi^2 m^2} +\frac{1}{\pi^2 m^2}\int_{-1}^1 w''(t)\exp(-i\pi mt)\, {\rm d}t\,.\\
\endaligned$$
The upper bound of the lemma follows.

Concerning the two first identities, they are immediately implied by standard calculus and the normal convergence of the Fourier series of $\tilde w$ which allow to interchange 
summation and integration. The third one follows from the previous two.
\end{proof}

The second lemma is of an arithmetical nature. In the course of proving the Theorem, we shall meet the following quantity
$$
S(\Acal)=\frac{1}{2N}\sum_{n=-N}^{N-1}\left(\left\vert\hat f\left(\frac{n}{2N}\right)\right\vert^2-\vert\Acal\vert\right)^2,
$$
where we use the notation of Section \ref{sec2}, in particular $\Acal$ is a set of integers included in $\{0,1,\dots,N\}$. 
The next lemma is an upper bound for $S(\Acal)$.

\begin{lem} \label{S} 
If $\Acal$ is a $B_2[g]$ set included in $\{0,1,\dots,N\}$, then
$$
S(\Acal)\le (2g-1)\vert\Acal\vert^2.
$$
\end{lem}

\begin{proof} 
It is for instance an intermediary result in the proof of Lemma 3 in \cite{Yu1}. The inequality follows from the fact that $S(\Acal)$
counts the number of solutions to the equation $a-b=c-d$ with $a,b,c,d \in \Acal$ and $a \neq b$.
\end{proof}

\section{Proof of the Theorem}

We now come to the central estimate of this article which is an explicit version of Lemma 2.2 of \cite{Yu2}. With such a result, we can apply the method 
not only to the case of polynomials but also to the case of power series.

\begin{thm}
\label{Maj} 
Let $\Acal$ be a $B_2 [g]$ set contained in $\{0,1,\dots,N\}$ and $b$ be an admissible function.
We have
\begin{eqnarray*}
D_\Acal (b)	&	\leq 	& 	\left(I_1(w_b)+\frac{A(w_b)}{4N^2}\right) \vert\Acal\vert^2 + \left(w_b(0)-I_1(w_b)\right)\vert\Acal\vert \\
	&		&	+\left(\sqrt{2\left(I_2(w_b)-I_1(w_b)^2\right)}+\frac{A(w_b)}{2 N^{3/2}}\right) \sqrt{ (2g-1)N\vert\Acal\vert^2-\frac{\vert\Acal\vert^4}{2}+\vert\Acal\vert^3}\,.
\end{eqnarray*}
\end{thm}

\begin{proof}[Proof of Theorem \ref{Maj}] We apply Lemma \ref{F} to $w_b$ (which is $C^2$ and even) and use the notation introduced there for ${\tilde w}_b$. By formula 
\eqref{ffbar}, we find
\begin{eqnarray*}
D_\Acal(b)	&	=	&  \sum_{\vert n\vert \leq N} d(n) w_b \left(\frac{n}{N}\right) \\
	&	=	&  \sum_{\vert n\vert \leq N} d(n) {\tilde w}_b \left(\frac{n}{N}\right) \\
	&	=	&  \sum_{\vert n\vert \leq N} d(n) \left( \frac{a_0}{2}+\sum_{m=1}^{+\infty} a_m\cos \left( \frac{\pi mn}{N} \right) \right) \\
	&	=	&  \frac{a_0}{2} |\Acal |^2 + \sum_{m=1}^{+\infty} a_m \sum_{\vert n\vert \leq N} d(n) \cos \left( \frac{\pi mn}{N} \right)  \\
	&	=	&  \frac{a_0}{2}\left\vert\hat f(0)\right\vert^2+ \sum_{m=1}^{+\infty}a_m \left\vert\hat f\left(\frac{m}{2N}\right)\right\vert^2 \\
	&	=	&  \frac{1}{2}\sum_{n=-\infty}^{+\infty}a_n \left\vert\hat f\left(\frac{n}{2N}\right)\right\vert^2\\
	&	=	&  w_b(0)\vert \Acal\vert + \frac{1}{2}\sum_{n=-\infty}^{+\infty}a_n \left(\left\vert\hat f\left(\frac{n}{2N}\right)\right\vert^2 - \vert \Acal\vert \right)\\
	&	= 	& w_b(0)\vert \Acal\vert + \frac{1}{2}\sum_{n=-N}^{N-1} \left(a_n+\sum_{k=1}^{\infty} a_{n+2kN} +
\sum_{k=1}^{\infty} a_{n-2kN}\right) \left(\left\vert\hat f\left(\frac{n}{2N}\right)\right\vert^2 - \vert \Acal\vert \right).\\
\end{eqnarray*}
Such a rearrangement of the terms of the series is allowed by the fact it is normally convergent. This follows from 
the bounds on the $a_m$ given by Lemma \ref{F} and the boundedness of the 
terms $| \hat{f}(n/2N)|$ (which are upper bounded by $| \Acal |$).

Restarting from this identity on $D_\Acal(b)$, we obtain (on recalling $\hat f\left(0 \right)=| \Acal|\geq 1$)
\begin{eqnarray*}
D_\Acal(b)	&	=	& w_b(0)\vert \Acal\vert + \frac{1}{2}\sum_{n=-N}^{N-1} \left(a_n+\sum_{k=1}^{\infty} a_{n+2kN} +
						\sum_{k=1}^{\infty} a_{n-2kN}\right) \left(\left\vert\hat f\left(\frac{n}{2N}\right)\right\vert^2 - \vert \Acal\vert \right)\\
			&	=	& w_b(0)\vert \Acal\vert + \frac{1}{2} \left(a_0+\sum_{k=1}^{\infty} a_{2kN} +
						\sum_{k=1}^{\infty} a_{-2kN}\right) \left(\vert \Acal \vert^2 - \vert \Acal\vert \right) \\
			&&			+ \frac{1}{2}\sum_{n=-N,\dots,N-1\atop n\neq0} \left(a_n+\sum_{k=1}^{\infty} a_{n+2kN} +
						\sum_{k=1}^{\infty} a_{n-2kN}\right) \left(\left\vert\hat f\left(\frac{n}{2N}\right)\right\vert^2 - \vert \Acal\vert \right)		%\\
\end{eqnarray*}			
and then			
\begin{eqnarray*}			
D_\Acal(b)	&	\leq	& w_b(0)\vert \Acal\vert + \frac{1}{2} \left(a_0+ \left| \sum_{k=1}^{\infty} a_{2kN} \right| +
						\left| \sum_{k=1}^{\infty} a_{-2kN} \right| \right) \left(\vert \Acal \vert^2 - \vert \Acal\vert \right) \\
			&&			+ \frac{1}{2}\sum_{n=-N,\dots,N-1\atop n\neq0} \left(|a_n|+\left| \sum_{k=1}^{\infty} a_{n+2kN} \right| +
						\left| \sum_{k=1}^{\infty} a_{n-2kN}\right| \right) \left| \left\vert\hat f\left(\frac{n}{2N}\right)\right\vert^2 - \vert \Acal\vert \right|	.
\end{eqnarray*}

By the upper bound given in Lemma \ref{F}, we have
$$
\left\vert\sum_{k=1}^{\infty}a_{n+2kN} \right\vert,\    \left\vert\sum_{k=1}^{\infty}a_{n-2kN} \right\vert   \le \sum_{k=1}^{\infty} \frac{2A(w_b)}{\pi^2((2k-1)N)^2}= \frac{A(w_b)}{4N^2}
$$
%and
%$$  
%\left\vert\sum_{k=1}^{\infty}a_{n-2kN} \right\vert\le \sum_{k=1}^{\infty} \frac{2A(w_b)}{\pi^2((2k-1)N)^2}= \frac{A(w_b)}{4N^2}
%$$
for $n=-N,\dots,N-1$.  We thus deduce
\begin{eqnarray*}
D_\Acal(b) 	&	\le	& 	w_b(0) \vert \Acal\vert +\left(\frac{a_0 }{2}+\frac{A(w_b)}{4N^2}\right)\left( \vert\Acal\vert^2-\vert\Acal\vert\right) \\
	&		&	\hspace{2cm}	+\frac{1}{2}\sum_{n=-N,\dots,N-1\atop n\neq0}\left(\vert a_n\vert+\frac{A(w_b)}{2N^2}\right) 
								\left\vert\left\vert\hat f\left(\frac{n}{2N}\right)\right\vert^2-\vert\Acal\vert\right\vert \,.
\end{eqnarray*}

The last term of this upper bound is bounded above using Cauchy-Schwarz inequality, more precisely
\begin{eqnarray*}
	\sum_{n=-N,\dots,N-1\atop n\neq0} \left(\vert a_n\vert+\frac{A(w_b)}{2N^2}\right) \left\vert\left\vert\hat f\left(\frac{n}{2N}\right)\right\vert^2-\vert\Acal\vert\right\vert \\
&\hspace{-14cm}=& \hspace{-7cm}\sum_{n=-N,\dots,N-1\atop n\neq0} \vert a_n\vert \left\vert\left\vert\hat f\left(\frac{n}{2N}\right)\right\vert^2-\vert\Acal\vert\right\vert
+ \frac{A(w_b)}{2N^2} \sum_{n=-N,\dots,N-1\atop n\neq0}	\left\vert\left\vert\hat f\left(\frac{n}{2N}\right)\right\vert^2-\vert\Acal\vert\right\vert \\
&\hspace{-14cm}	\le 	&\hspace{-7cm}	\left( \left( \sum_{n=-N,\dots,N-1\atop n\neq0} \vert a_n\vert^2\right)^{1/2} + \frac{A(w_b)}{2N^2} (2N-1)^{1/2} \right)
\left( \sum_{n=-N,\dots,N-1\atop n\neq0} \left\vert\left\vert\hat f\left(\frac{n}{2N}\right)\right\vert^2-\vert\Acal\vert\right\vert^2\right)^{1/2}\\
&\hspace{-14cm}	\le 	&\hspace{-7cm}	\left( \left( \sum_{n=-N,\dots,N-1\atop n\neq0} \vert a_n\vert^2\right)^{1/2} + \frac{A(w_b)}{\sqrt{2} N^{3/2} } \right)
\left( \sum_{n=-N,\dots,N-1\atop n\neq0} \left\vert\left\vert\hat f\left(\frac{n}{2N}\right)\right\vert^2-\vert\Acal\vert\right\vert^2\right)^{1/2}\,.
%\left( \sum_{n=-N,\dots,N-1\atop n\neq0} \left\vert\left\vert\hat f\left(\frac{n}{2N}\right)\right\vert^2-\vert\Acal\vert\right\vert^2\right)^{1/2}
\end{eqnarray*}

Plugging this bound, we finally obtain 
\begin{eqnarray*}
D_\Acal(b)	&	\le 	&	w_b(0)\vert \Acal\vert+\left(\frac{a_0 }{2}+\frac{A(w_b)}{4N^2}\right)\left( \vert\Acal\vert^2-\vert\Acal\vert\right)\\
	&		&	+\frac{1}{2}\left(\left( \sum_{n=-N,\dots,N-1\atop n\neq0} \vert a_n\vert^2\right)^{1/2}+\frac{A(w_b)}{\sqrt2 N^{3/2}}\right)
\left( \sum_{n=-N,\dots,N-1\atop n\neq0} \left\vert\left\vert\hat f\left(\frac{n}{2N}\right)\right\vert^2-\vert\Acal\vert\right\vert^2\right)^{1/2}\\
	&	\le	&	 w_b(0)\vert \Acal\vert+\left(\frac{a_0 }{2}+\frac{A(w_b)}{4N^2}\right)\left( \vert\Acal\vert^2-\vert\Acal\vert\right)\\
	&		&	+\left(\left( \frac{1}{2}\sum_{n=-N,\dots,N-1\atop n\neq0} \vert a_n\vert^2\right)^{1/2}+\frac{A(w_b)}{2 N^{3/2}}\right)
\left( N\Scal(\Acal) - \frac{(\vert\Acal\vert^2-\vert\Acal\vert)^2}{2}\right)^{1/2}\,.
\end{eqnarray*}

It is now enough to use the final identities of Lemma \ref{F}. Since $I_1(w_b)=a_0/2$ and
$$
\frac{1}{2}\sum_{n=-N,\dots,N-1\atop n\neq0}  \vert a_n\vert^2\le \sum_{n=1}^{\infty} \vert a_n\vert^2=2(I_2(w_b)-I_1(w_b)^2),
$$
we conclude using Lemma \ref{S}.
\end{proof}

\begin{cor}
\label{tau1}
Let $b$ be an arbitrary admissible function such that $I_1(w_b)<0$, then we have
$$
\limsup_{N\to\infty} \frac{F(g,N)^2}{(2g-1)N} \le 2 \left(1-\frac{I_1(w_b)^2}{I_2(w_b)}\right) .
$$
\end{cor}

\begin{proof} 
Let $\Acal$ be a $B_2[g]$ set in $\{1,\dots,N\}$ with $|\Acal | = F(g,N)$. In particular, $|\Acal | \gg \sqrt{N}$.

For an arbitrary admissible function $b$, we apply Theorem  \ref{Maj} and let $N$ tend to infinity. Since, by non-negativity of $b$ and inequality \eqref{D}, $D_\Acal (b) \geq 0$, we obtain
$$
\left(-I_1(w_b)+o(1)\right)\vert\Acal\vert^2\le  \left(\sqrt{2\left(I_2(w_b)-I_1(w_b)^2\right)}+o(1)\right)
\sqrt{ (2g-1)N\vert\Acal\vert^2-\frac{\vert\Acal\vert^4}{2}+\vert\Acal\vert^3}\,.
$$
Thanks to the assumption that $I_1(w_b)<0$, one can square the preceding inequality and we obtain 
$$
\left(I_1(w_b)^2+o(1)\right)\vert\Acal\vert^4\le 2\left(I_2(w_b)-I_1(w_b)^2\right)
\left((2g-1)N\vert\Acal\vert^2-\frac{\vert\Acal\vert^4}{2}\right)
$$
and, after simplification,
$$
\left(I_2(w_b)+o(1)\right)\vert\Acal\vert^2 \le 2 \left(I_2(w_b)-I_1(w_b)^2\right)(2g-1)N.
$$
The corollary follows.
\end{proof}

\section{Optimization : Choosing $b$}

In view of Corollary \ref{tau1}, we are led to the optimization problem of computing 
$$
\max_{b\ {\rm admissible\ such\ that\ }  I_1(w_b)<0}  \hspace{1cm}  \frac{I_1(w_b)^2}{I_2(w_b)}.
$$

In his paper \cite{Yu2}, Yu first (his Theorem 1) chooses the function
$$
w(t)=\sum_{m=0}^M \frac{ \cos \left( 2 \pi (m+\lambda) t \right) }  {m+\lambda}
$$
where $M$ is taken equal to $10^6$ and $\lambda= 3/4$ (a case for which computations are made easier)  which gives the bound $1.74246$.
Yu then proceeds with a numerical optimization and finally, with $\lambda = 0.75315$, he gets the value $1.74217$ leading to Yu's second theorem (this is the value 
mentioned in \eqref{upperbound1}). 

There are several ways to improve on this result. First, our general result can be applied to any truncation of the infinite series (which is non convergent for $t=0$) 
associated with Yu's function. 
If we go back to the case $\lambda = 3/4$, and let $M$ tend to infinity, this already gives the bound $1.7424537\dots$, %$45422$, 
which is the limit of Yu's function with this choice of parameter. But, again, one may then move slightly $\lambda$. We used a signed continued fraction method 
which leads us to consider the value $\lambda = 365/478$ (at some step). With this choice of $\lambda$, we are led 
to the numerical upper bound $1.7407259\dots$ %237377\dots$ 
We do not enter into more details here since this method does not give the best value we could obtain.

In fact, there is no reason to choose such a regular function $w$. We started a numerical study on functions of the form 
$$
w(t) = \cos \big( (y_0+ \pi) t \big) + \sum_{j=1}^M \frac{c_j}{j}  \cos \big( (y_j+(2j+1)\pi ) t \big).
$$

We used a Maple program and could go up to $M=400$  (that is, $801$ variables). 
The computation took about four days on a shared machine equipped
with two Intel Xeon E5-2470v2 processors. 
Notice that, more than time-consuming, this approach is very
space-consuming and in fact limited by space considerations.
%2 processeurs Intel Xeon E5-2470v2 (25M Cache, 2.40GHz,8 GT/s QPI)
%avec 10 coeurs hyperthreaded => 40 CPU estimation : 192026.60 BogoMIPS    96Go de RAM
In the above form, we were looking for an optimum where the $y_j$ are restricted to belong to $(0,\pi)$ and the $c_j$ to $(0,1)$. 
It turns out that when we increase the number $M$ of variables, the values $y_j$ and $c_j$ seem to converge. Here are the first values 
that are given by the optimization process (obtained for $M=400$):
\begin{eqnarray*}
c_{01 } = 0.448668493767477, &c_{02 } = 0.575146465019734,& c_{03 } = 0.634139353767643,\\
c_{04 } = 0.668206769165044, &c_{05 } = 0.690373909392123,& c_{06 } = 0.705944152178521,\\
    c_{07 } = 0.717479053644182,& c_{08 } = 0.726366349898625,& c_{09 } = 0.733423759607086,\\ 
    c_{10 } = 0.739163465377496,& c_{11 } = 0.743922783065952,& c_{12 } = 0.747933037687434,\\
    c_{13 } = 0.751358473065359,& c_{14 } = 0.754318115226197,& c_{15 } = 0.756900829824045,\\
    c_{16 } = 0.759174482613027,& c_{17 } = 0.761190238930946,& c_{18 } = 0.762988657959701,\\
c_{19 } = 0.764605831570057,& c_{20 } = 0.766063873483719,& c_{21 } = 0.767398988945215, \\
    c_{22 } = 0.768616037123302,& c_{23 } = 0.769721510942451,& c_{24 } = 0.770739989883381,\\
    c_{25 } = 0.771678878036841,& c_{26 } = 0.772543457251216,& c_{27 } = 0.773353319988327, \\
    c_{28 } = 0.774096401927810,& c_{29 } = 0.774802358105814,& c_{30 } = 0.775461565599078,\\
    c_{31 } = 0.776070438424819,& c_{32 } = 0.776640535845029,& c_{33 } = 0.777213408942223, \\
    c_{34 } = 0.777688024987857,& c_{35 } = 0.778162522583045,& c_{36 } = 0.778618081806088,\\
    c_{37 } = 0.779075729278605,& c_{38 } = 0.779444959637105,& c_{39 } = 0.779857433648994, \\
    c_{40 } = 0.780247031029276,& c_{41 } = 0.780579370448116,& c_{42 } = 0.780921813816887,\\
    c_{43 } = 0.781221129831046,& c_{44 } = 0.781554783493105,& c_{45 } = 0.781870431056320, \\
    c_{46 } = 0.782110198962599,& c_{47 } = 0.782361619824327,& c_{48 } = 0.782643557927602,\\
    c_{49 } = 0.782885035586508,& c_{50 } = 0.783100192717692,&
\end{eqnarray*}
and 
\begin{eqnarray*}
y_{00 } = 1.69023069423400,& y_{01 } = 1.62455004938005,& y_{02 } = 1.60400691427448,\\
    y_{03 } = 1.59374507362384,& y_{04 } = 1.58739065526372,& y_{05 } = 1.58292851285127,\\
    y_{06 } = 1.57952428074446,& y_{07 } = 1.57677070519547,& y_{08 } = 1.57444556939989,\\
    y_{09 } = 1.57241834643895,& y_{10 } = 1.57060466311460,& y_{11 } = 1.56895032673690,\\
    y_{12 } = 1.56741998541706,& y_{13 } = 1.56598348343700,& y_{14 } = 1.56462349022195,\\
    y_{15 } = 1.56332531960606,& y_{16 } = 1.56207725583259,& y_{17 } = 1.56086642851363,\\
    y_{18 } = 1.55969722035216,& y_{19 } = 1.55855788286864,& y_{20 } = 1.55745188656436,\\
    y_{21 } = 1.55638570641239,& y_{22 } = 1.55528462446397,& y_{23 } = 1.55421905033814,\\
    y_{24 } = 1.55318764446397,& y_{25 } = 1.55213468181519,& y_{26 } = 1.55113576643217,
    \end{eqnarray*}
\begin{eqnarray*}
    y_{27 } = 1.55011416521470,& y_{28 } = 1.54911054412942,& y_{29 } = 1.54815575459570,\\
    y_{30 } = 1.54715785448177,& y_{31 } = 1.54615472793709,& y_{32 } = 1.54518383791521,\\
        y_{33 } = 1.54424768835177,& y_{34 } = 1.54324227742403,& y_{35 } = 1.54234694571695,\\
           y_{36 } = 1.54139048590958,& y_{37 } = 1.54036349157331,& y_{38 } = 1.53942606099970,\\
        y_{39 } = 1.53850611740410,& y_{40 } = 1.53758211330524,& y_{41 } = 1.53663231603874,\\
    y_{42 } = 1.53567473396147,& y_{43 } = 1.53474740944525,& y_{44 } = 1.53383628504159,\\
    y_{45 } = 1.53290791051452,& y_{46 } = 1.53193597506582,& y_{47 } = 1.53097247735348,\\
    y_{48 } = 1.53007947174410,& y_{49 } = 1.52921326776155,& y_{50 } = 1.52829122078524.
\end{eqnarray*}
The interested reader can refer to the complete numerical results available in \cite{HPnum}.
Notice that this function remains close to Yu's function, which after renormalization can be taken equal to 
$$
w(t)=\sum_{m=0}^M \frac{\lambda} {m+\lambda}  \cos \left( ((2m+1)\pi+ (2 \lambda-1)\pi ) t \right).
$$
Indeed the coefficients $c_j$ remains around $0.75$ while the coefficients $y_j$ are slightly above $1.5$.

Finally considering these values (and those for bigger indices) for $M=400$ led us to the value 
$1.74046270371931700$ and thus to Theorem \ref{theo}.

Heuristically, it seems that the method could be pushed up to proving the bound $1.74$. However, if true 
and provable by the present method, this could require to use a value of $M$ much larger than $400$.

\bigskip\bigskip

\end{document}